\documentclass[11pt]{amsart} \usepackage{latexsym, amssymb, mathrsfs,graphicx}
\usepackage{hyperref}

\newtheorem{thm}{Theorem}[section]
\newtheorem{lemma}[thm]{Lemma}
\newtheorem{prop}[thm]{Proposition}
\newtheorem{cor}[thm]{Corollary}
\newtheorem{question}[thm]{Question}
\newtheorem{fact}[thm]{Fact}

\theoremstyle{definition}
\newtheorem{df}[thm]{Definition}

\newtheorem{nrmks}[thm]{Remarks}

\theoremstyle{remark}

\renewcommand{\r}{\mathbb{R}}
\newcommand{\Z}{\mathbb{Z}}

\newcommand{\n}{\mathbb{N}}

\renewcommand{\to}{\rightarrow}

\def \rank{\operatorname{rank}}

\def \R { {\mathbb R} }

\def \<{\langle}
\def \>{\rangle}
\def \z {{\mathbb Z}}
\def \*Z {{{^*}\Z}}
\def \((  {(\!(}
\def \)) {)\!)}

\def \ns{\operatorname{ns}}

\def \f{\operatorname{fin}}

\numberwithin{equation}{section}

\def \st{\operatorname{st}}
\def \hyp{\operatorname{hyp}}
\def \ipc{\operatorname{IPC}}
\def \e{\operatorname{end}}
\def \ee{\operatorname{Ends}}
\def \id{\operatorname{id}}

\allowdisplaybreaks[2]

\begin{document}

\title[Hyperfinite Fundamental Group]{The Fundamental Group of a Locally Finite Graph with Ends- A Hyperfinite Approach}
\author{Isaac Goldbring}
\thanks{Goldbring's work was partially supported by NSF grant DMS-1007144.}
\thanks{The authors would like to thank the Centre International de Recontres Math\'ematiques in Luminy, France and the organizers of the meeting \emph{Model Theory of Groups}, which took place in November 2011, where the authors first met to discuss this project.}
\address{University of California, Los Angeles, Department of Mathematics\\ 520 Portola Plaza, Box 951555\\ Los Angeles, CA 90095-1555, USA}
\email{isaac@math.ucla.edu}
\urladdr{http://www.math.ucla.edu/~isaac}
\author{Alessandro Sisto}
\address{University of Oxford-Mathematical Institute\\ 24-29 St Giles\\ Oxford OX1 3LB \\United Kingdom}
\email{sisto@maths.ox.ac.uk}
\urladdr{http://people.maths.ox.ac.uk/sisto}

\begin{abstract}
The \emph{end compactification} $|\Gamma|$ of the locally finite graph $\Gamma$ is the union of the graph and its ends, endowed with a suitable topology. We show that $\pi_1(|\Gamma|)$ embeds into a nonstandard free group with hyperfinitely many generators, i.e. an ultraproduct of finitely generated free groups, and that the embedding we construct factors through an embedding into an inverse limit of free groups.  We also show how to recover the standard description of $\pi_1(|\Gamma|)$ given in \cite{diestel}.  Finally, we give some applications of our result, including a short proof that certain loops in $|\Gamma|$ are non-nullhomologous.
\end{abstract}

\maketitle

\section{Introduction}

It is a well-known fact that the fundamental group of a finite, connected graph $\Gamma$ (viewed as a one-dimensional CW-complex) is a finitely generated free group.  Indeed, fix a spanning tree $T$ of $\Gamma$ and let $e_1,\ldots,e_n$ denote the chords of $T$ (that is, edges of $\Gamma$ not in $T$) equipped with a fixed orientation.  Then to any loop $\alpha$, we can consider the word $r_\alpha$ on the alphabet $\{e_1^{\pm 1},\ldots,e_n^{\pm 1}\}$ obtained by recording the traversals of $\alpha$ on each $e_i$ keeping into account the orientation.
If we let $F_n$ denote the free group with basis $e_1,\ldots,e_n$ and we let $[r_\alpha]$ denote the unique reduced word corresponding to $r_\alpha$, then the map
$\alpha\mapsto [r_\alpha]:\pi_1(\Gamma)\to F_n$ is an isomorphism.

We now consider the case that $\Gamma$ is an infinite, locally finite, connected graph.  In order to obtain a compact space, we consider the \emph{end compactification} $|\Gamma|$ of $\Gamma$ obtained by adding the \emph{ends of $\Gamma$}.  Loosely speaking (we will make this more precise in the next section), the ends of $\Gamma$ are the ``path components of $\Gamma$ at infinity''; see the paper \cite{gold} where this heuristic is made precise using the language of nonstandard analysis.  For example, if $\Gamma$ is the Cayley graph of $\z$ with its usual generating set $\{1\}$, then $\Gamma$ has two ends, one at ``$-\infty$'' and one at ``$+\infty$.''  End compactifications of graphs have been considered, e.g., in \cite{diestel2,diestel3,DK-infcyc,Geo-cir1,Geo-cir2,BB-euler} as a way to obtain analogues for infinite graphs of results in finite graph theory that would otherwise be plainly false.  

In \cite{diestel}, the authors find a combinatorial characterization of $\pi_1(|\Gamma|)$ similar in spirit to the aforementioned construction for finite graphs.  The authors have to overcome two main obstructions, both of which are illustrated through the following simple example.  Consider the 1-way infinite sideways ladder $\Gamma$ indicated below in Figure \ref{ladder}.

\begin{figure}[h]
\label{ladder}
\includegraphics[scale=0.8]{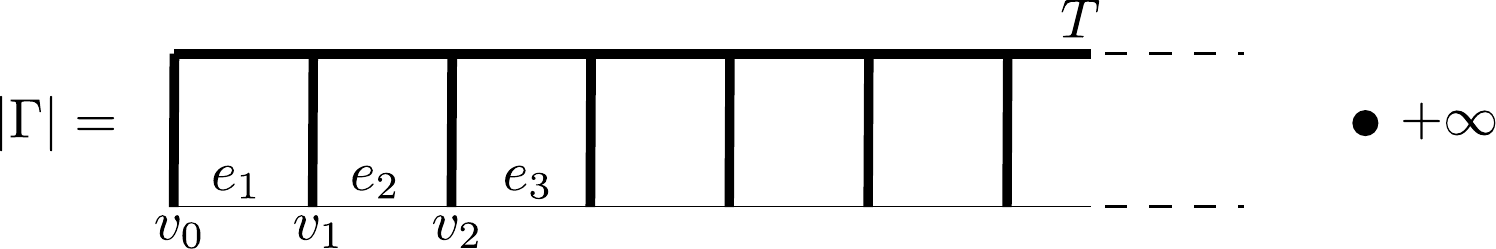}
\caption{}
\end{figure}
$\Gamma$ has one end, namely the point at ``$+\infty$.''  The loop $\alpha$ that starts at $v_0$, runs down the bottom edge of the ladder to $+\infty$ and then back to $v_0$ is clearly nullhomotopic.  However, if one looks at the traces $\alpha$ leaves on the chords of the spanning tree $T$ pictured in Figure \ref{ladder}, we are left with the following word:
$$r_\alpha=(e_1e_2\cdots) ^\frown (\cdots e_2e_1).$$  This word has order type $\omega+\omega^*$ (where $\omega$ is the order type of the natural numbers and $\omega^*$ is $\omega$ equipped with its reverse ordering) and contains no consecutive appearances of $e_i,e_i^{-1}$ for any $i$. In particular, the (usual) notion of reduction of words does not work well in this context.  More generally, for any infinite, locally finite graph $\Gamma$, the authors of \cite{diestel} consider the group $F_\infty$ consisting of reduced infinite words on the alphabet of oriented edges of a \emph{topological} spanning tree of $\Gamma$; here, an infinite word can have arbitrary countable order type, leading to a complicated (non-well-ordered) notion of reduction of words. Just taking any spanning tree of $\Gamma$ would not work, as the closure of such a tree in $|\Gamma|$ might contain loops. They then show that $\pi_1(|\Gamma|)$ embeds as a subgroup of $F_\infty$ and $F_\infty$ embeds as a subgroup of an inverse limit of finitely generated free groups by sending an infinite word to the family of its finite subwords. These are, respectively, part $(i)$ and $(ii)$ of \cite[Theorem 15]{diestel}.

After seeing the nonstandard approach to ends outlined in \cite{gold}, Diestel asked the first author whether or not the nonstandard approach could lead to a simplification of the article \cite{diestel}.  Immediately the following idea arose:  Consider the graph $\Gamma$ from Figure \ref{ladder} and fix an infinite natural number $\nu\in \n^*\setminus \n$.  (In this paper, we will assume that the reader is familiar with the basics of nonstandard analysis; otherwise they may consult  \cite{D} or \cite{He}.)  We then consider the \emph{hyperfinite} graph $\Gamma_\nu$ appearing in Figure \ref{gammanu}.

\begin{figure}[h]
\includegraphics[scale=0.8]{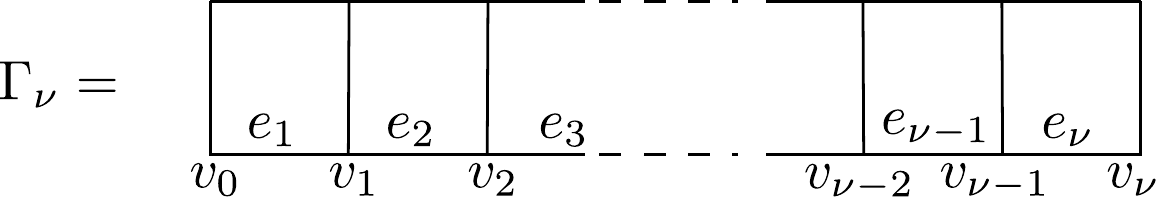}
\caption{}
\label{gammanu}
\end{figure}
The nullhomotopic loop $\alpha$ from the previous paragraph ``naturally'' induces an \emph{internal} loop in $\Gamma^*$, namely the loop that starts at $v_0$, travels down the bottom edge to $v_\nu$, then returns back to $v_0$.  The \emph{hyperfinite} word that $\alpha$ corresponds to is $e_1e_2\cdots e_\nu e_{\nu}^{-1}\cdots e_2^{-1}e_1^{-1}$, which clearly reduces to the trivial word, witnessing that $\alpha$ is nullhomotopic.

The main purpose of this article is to make precise the na\"ive approach taken in the previous paragraph.  More specifically, we prove

\begin{thm}\label{main}
Suppose $\Gamma$ is a locally finite connected graph.  Then there is a hyperfinite, internally connected graph $\Gamma_{\hyp}$ and an \emph{injective} group morphism $\Theta:\pi_1(|\Gamma|)\hookrightarrow\pi_1(\Gamma_{\hyp})$, where $\pi_1(\Gamma_{\hyp})$ is the \emph{internal fundamental group of $\Gamma_{\hyp}$}. Moreover, $\Theta$ can be constructed in such a way that it factors through an embedding of $\pi_1(|\Gamma|)$ into an inverse limit of finitely generated free groups.
\end{thm} 

By the transfer principle, $\pi_1(\Gamma_{\hyp})$ will be an \emph{internally free group on hyperfinitely many generators}.  Said another way:  $\pi_1(|\Gamma|)$ embeds into an ultraproduct of finitely generated free groups. 
The ``moreover'' part gives part of the aforementioned result of \cite{diestel}.  In fact, in Section \ref{recover}, we will show how our methods can be used to completely recover part $(i)$ of \cite[Theorem 15]{diestel} as described above (part $(ii)$ is purely algebraic and the nonstandard set-up has nothing to say about this portion of the theorem). Two sample consequences of the theorem are reported below.

\begin{cor}
\ 
\begin{enumerate}
 \item $\pi_1(|\Gamma|)$ has the same universal theory as a free group.
 \item $\pi_1(|\Gamma|)$ is $\omega$-residually free.
\end{enumerate}
\end{cor}

Recall that a group $G$ is \emph{$\omega$-residually free} if for any $g_1,\ldots,g_n\in G\setminus\{1\}$, there is a surjective homomorphism $\sigma:G\to F$ onto a free group such that $\sigma(g_i)\not=1$ for $i=1,\ldots,n$.  For finitely generated groups, conditions (1) and (2) in the above corollary are well-known to be equivalent, see e.g. \cite{CG}. Finitely generated groups satisfying conditions (1) and (2) are called \emph{limit groups} and are widely studied.  

We should mention that our construction is simpler than that of \cite{diestel} in that we do not need to concern ourselves with topological spanning trees (whose existence is nontrivial).  On the other hand, our construction uses the \emph{universal cover} of $\Gamma_{\hyp}$, which is an internal tree where homotopies can be reparameterized in nice ways.  Only with the hyperfinite approach is the recourse to the universal covering tree possible.
Some applications of this result (and the construction of $\Theta$) will be given in Section \ref{appl}, where we give another proof of the fact that any inclusion of locally finite connected graphs $\Gamma^1\to \Gamma^2$ induces an injection $\pi_1(|\Gamma^1|)\to \pi_1(|\Gamma^2|)$, the main reason being that the analogous result is true for finite graphs.

The original motivation for the work in \cite{diestel} was to show that certain loops in $|\Gamma|$ are non-nullhomologous. The motivation for this is that these loops are trivial in the so-called topological cycle space, another homology theory that has proven to be more suitable than the usual one to study Fredudenthal compactifications (of spaces even more general than graphs) \cite{DiSp-hom,DS-loccpthom}. The existence of such loops implies that these theories are actually different. Using our techniques and the algorithm for calculating commutator lengths in finitely generated free groups, in Section \ref{appl} we give a very short proof that one of the loops in the infinite-sideways ladder above witnessing the difference between the said homology theories is in fact non-nullhomologous. We also make the remark that a loop $\alpha$ gives rise to the trivial element of the topological cycle space if and only if $\Theta(\alpha)$, for $\Theta$ as in the main theorem, is internally null-homologous in $H_1(\Gamma_{\hyp})$ (with coefficients in $\z/2\z$ or $\z$ depending on the version of topological cycle space under consideration).

We want to caution that our result does not imply that $\pi_1(|\Gamma|)$ is a free group.  (According to Corollary 18 of \cite{diestel}, $\pi_1(|\Gamma|)$ is free if and only if every end of $\Gamma$ is contractible in $|\Gamma|$.)  In general, \emph{an internally free group is not actually free}.  For example, $\z^*$ is internally free on one generator.  Fix $M,N\in \n^*\setminus \n$ with $\frac{M}{N}$ infinitesimal and consider the injective group morphism $(a,b)\mapsto aM+bN:\z^2\to \z^*$.  Since $\z^2$ is not free, we see that $\z^*$ is not free.



We would like to thank Reinhard Diestel for suggesting this project to us and for being so supportive of the nonstandard approach.

\section{Preliminaries on End Compactifications}

In this section, we fix a proper, pointed geodesic metric space $(X,d,p)$.  For $n\in \n$, we let $B(p,n)$ (resp. $\mathring{B}(p,n)$) denote the closed (resp. open) ball of radius $n$ around $p$.  For $x,y\in X$, we write $x\propto_n y$ to mean that there is a path in $X$ connecting $x$ to $y$ which remains entirely in $X\setminus B(p,n)$.  Given proper rays $r_1,r_2:[0,\infty)\to X$ based at $p$, we say that $r_1$ and $r_2$ \emph{determine the same end of $X$} if for every $n\in \n$, there is $t_0\in [0,\infty)$ such that for all $t\geq t_0$, $r_1(t)\propto_n r_2(t)$.  The notion of determining the same end is an equivalence relation on the set of proper rays beginning at $p$ and we let $\e(r)$ denote the equivalence class of the proper ray $r$.

The \emph{end compactification of $X$}, denoted $|X|$, is the set $X$ together with its ends.  There is a natural topology on $|X|$, where a neighborhood basis of $x\in X$ is a neighborhood basis of $x$ induced by the metric and a neighborhood basis for $\e(r)$ is given by sets of the form
$$W_n(\e(r))=\{\e(r')\ |\ \exists m_0\in\n\ \forall m\geq m_0 \ r(m)\propto_n r'(m)\}\cup$$
$$ \{x\in X\ | \ \exists m_0\in\n\ \forall m\geq m_0\ r(m)\propto_n x\}.$$

In \cite{gold}, a nonstandard characterization of the topology of the space of ends of $X$ was outlined.  In this section, we adapt the arguments from \cite{gold} to show that $|X|$ is the same as the space $|X|_{\ns}$ (the ``nonstandard ends compactification'') defined below. First, set $$X_{\f}:=\{x\in X^* \ | \ d(x,p) \text{ is finite}\} \quad \text{ and } \quad X_{\inf}:=X^*\setminus X_{\f}.$$  For $x,y\in X^*$, write $x\propto y$ (resp. $x\propto_n y$ for some $n\in\n$) if there exists an internal path connecting $x$ to $y$ in $X^*\backslash X_{\f}$ (resp. in $X^*\backslash B(p,n)$).  For $x\in X_{\inf}$, let $[x]$ denote the equivalence class of $x$ with respect to $\propto$. Set $\ipc(X):=\{[x] \ | \ x\in X_{\inf}\}$ and $|X|_{\ns}:=X\cup \ipc(X)$.  We consider the topology on $|X|_{\ns}$ where a neighborhood basis for $[x]\in \ipc(X)$ is given by
$$V_n([x]):=\{[x']\in \ipc(X)\ | \ x'\propto_n x\}\cup \{x'\in X\ | \ x'\propto_n x\},$$
and a neighborhood basis for $x\in X$ is a neighborhood basis for the topology induced by the metric. \

In \cite{gold}, it is proven that $\ipc(X)$ is homeomorphic to the space $\ee(X)$ of the ends of $X$ via the homeomorphism $\eta:\ee(X)\to \ipc(X)$ given by $\eta(\e(r))=[r(\sigma)]$, where $\sigma\in \r_{\inf}^{>0}$ is fixed. Extend $\eta$ to a map $\overline{\eta}:|X|\to |X|_{ns}$ by requiring $\overline{\eta}$ to be the identity map on $X$.

\begin{prop}
 $\overline{\eta}:|X|\to |X|_{\ns}$ is a homeomorphism.
\end{prop}
 
\begin{proof}
We begin with the following

\noindent \textbf{Claim:}  $W_n(\e(r))=W'_n(\e(r))$, where
$$W'_n(\e(r))=\{\e(r')\ |\ r(\sigma)\propto_n r'(\sigma)\}\cup \{x\in X\ | \ x\propto_n r(\sigma)\}.$$

\noindent \textbf{Proof of Claim:}  The containment $\subseteq$ follows by transfer as $\sigma\geq m_0$.  For the other containment, suppose that $\e(r')\in W_n'(\e(r))$, so $r(\sigma)\propto_n r'(\sigma)$.  If $\e(r')\notin W_n(\e(r))$, then, by overflow, there is $\sigma'>\n$ such that $r(\sigma')\not\propto_n r'(\sigma')$.  Since $r(\sigma)\propto_n r(\sigma')$ and $r'(\sigma)\propto_n r'(\sigma')$, we get $r(\sigma)\not\propto r'(\sigma)$, a contradiction.  The other case is similar.

\

\noindent Since $|X|$ is compact and $\overline{\eta}$ is bijective, it suffices to show that $\overline{\eta}$ is continuous.  $\overline{\eta}$ is clearly continuous on $X$, so let us consider $\e(r)\in |X|$ and $n>0$.  The proof will be finished if we can show that $\overline{\eta}(W^{'}_n(\e(r)))
=V_n([r(\sigma)])$.  However, this is immediate from the definitions and the fact that $\eta$ is a homeomorphism.
\end{proof}

\begin{lemma}
If $|X^*|$ denotes the internal ends compactification of $X^*$, then $|X|^*=|X^*|$. 
\end{lemma}

\begin{proof}
Both spaces are the union of $X^*$ and the quotient of the internally proper rays in $X^*$ emanating from $p$ modulo the relation of determining the same internal end. 
\end{proof}

\section{Embedding $\pi_1(\Gamma)$ into $\pi_1(\Gamma_{hyp})$}
\label{emb}

\subsection{The Main Theorem}

In this section, we fix an infinite, locally finite, connected graph $\Gamma$ with end compactification $|\Gamma|$.  We also fix a basepoint $p\in \Gamma$.

\begin{lemma}\label{ext}
Suppose that $n\in \n$ and $\theta:\Gamma \to \Gamma_n$ is the map which collapses each connected component of $\Gamma\setminus \mathring{B}(p,n)$ to a point.  Then $\theta$ extends continuously to a map $\theta:|\Gamma|\to \Gamma_n$.
\end{lemma}

\begin{proof}
Given $\e(r)\in |\Gamma|$ and $x\in \Gamma\cap W_n(\e(r))$, set $\theta(\e(r)):=\theta(x)$.  It is then easy to check that $\theta(W_n(r))=\{\theta(x)\}$, whence $\theta$ is continuous at $\e(r)$.
\end{proof}

In the rest of this section, we fix $\nu\in \n^*\setminus \n$ and set $\theta:\Gamma^*\to \Gamma_{\hyp}$ to be the map which collapses each internally connected component of $\Gamma^*\setminus \mathring{B}(p,\nu)$ to a point.  We consider $\Gamma$ both as a subset of $|\Gamma|^*$ and $\Gamma_{\hyp}$ in the obvious way and let $\hat{\Gamma}$ denote both the elements of $|\Gamma|^*$ and $\Gamma_{\hyp}$ an infinitesimal distance away from an element of $\Gamma$.  (This double use of $\hat{\Gamma}$ should not cause any confusion.) By the previous lemma, we obtain an extension of $\theta$ to an internally continuous map $\theta:|\Gamma^*|=|\Gamma|^*\to \Gamma_{\hyp}$.  Consequently, we get an induced map on the internal fundamental groups:

\begin{df}
We set $\Theta:\pi_1(|\Gamma|^*)\to \pi_1(\Gamma_{\hyp})$ to be the group homomorphism induced by the map $\theta:|\Gamma|^*\to \Gamma_{\hyp}$.  Here, $\pi_1(|\Gamma|^*)$ and $\pi_1(\Gamma_{\hyp})$ are the internal fundamental groups of $|\Gamma|^*$ and $\Gamma_{\hyp}$ respectively.
\end{df}

Observe that $\pi_1(|\Gamma|^*)=(\pi_1(|\Gamma|))^*$, so $\pi_1(|\Gamma|)$ is a subgroup of $\pi_1(|\Gamma|^*)$.  Our main goal is to show that $\Theta|_{\pi_1(|\Gamma|)}$ is injective, proving Theorem \ref{main}.  First, we show that the we can construct $|\Gamma|$ in a nonstandard fashion analogous to the construction of Section 2 using $\Gamma_{\hyp}$ instead of $\Gamma_{\inf}$. Indeed, we will later need a slight generalization of this: Consider an internally connected hyperfinite graph $\Lambda$ and internally continuous maps $\varphi:|\Gamma|^*\to \Lambda$, $\psi:\Lambda\to \Gamma_{\hyp}$, satisfying:
\begin{itemize}
 \item $\psi\circ\varphi=\theta$,
 \item $\varphi(\hat{\Gamma})=\psi^{-1}(\hat{\Gamma})$.
\end{itemize}

We remark that, for the purposes of this section, we will take $\Lambda=\Gamma_{\hyp}$, $\varphi=\theta$, and $\psi=\id_{\Gamma_{\hyp}}$. In what follows, we will also regard $\Gamma$ as a subset of $\Lambda$ and once again let $\hat{\Gamma}$ denote the elements of $\Lambda$ an infinitesimal distance away from an element of $\Gamma$.  We let $d$ denote the usual metric on $\Gamma$, namely the path metric on $\Gamma$ when each edge is identified with the interval $[0,1]$.  Similarly, we let $d_h$ denote the internal metric on $\Lambda$ and $\Gamma_{\hyp}$ obtained by identifying each edge with $[0,1]^*$.

For $x,y\in \Lambda$, write $x\propto y$ (resp. $x\propto_n y$ for some $n\in\n$) if there exists an internal path connecting $x$ to $y$ in $\Lambda\backslash \hat{\Gamma}$ (resp. in $\Lambda\backslash B(n,p)$), and denote by $[x]$ the equivalence class of $x$ with respect to $\propto$. Also, let $\ipc_{\hyp}(\Gamma)$ denote the collection of such equivalence classes. Finally, consider the topology on $|\Gamma|_{\hyp}:=\Gamma\cup \ipc_{\hyp}(\Gamma)$ where a neighborhood basis for $[x]\in \ipc_{\hyp}(\Gamma)$ is given by
$$V^{\hyp}_n([x]):=\{[x']\in \ipc_{\hyp}(\Gamma)\ | \ x'\propto_n x\}\cup \{x'\in\Gamma\ | \ x'\propto_n x\},$$
and a neighborhood basis for $x\in\Gamma$ is a neighborhood basis for the topology induced by the metric.
\par
Suppose $x,y\in \Gamma_{\inf}$ are such that there is an internal path connecting $x$ to $y$ in $\Gamma^*\backslash \hat{\Gamma}$.  Then, composing with $\varphi$, we get an internal path connecting $\varphi(x)$ to $\varphi(y)$ in $\Lambda\backslash \hat{\Gamma}$.  This allows us to define a map $\chi:|\Gamma|_{\ns}\to |\Gamma|_{\hyp}$ which restricts to the identity on $\Gamma$ and satisfies $\chi([x])=[\varphi(x)]$ for each $x\in \Gamma_{\inf}$.
\begin{lemma}\label{chi}
 $\chi$ is a homeomorphism.
\end{lemma}

\begin{proof}
Arguing as before the lemma, $V_n([x])$ is mapped into $V^{\hyp}_n([\varphi(x)])$ for each $x\in \Gamma_{\inf}$, whence $\chi$ is continuous.  Since $|\Gamma|_{\ns}$ is compact, it suffices to prove that $\chi$ is bijective.  Since $\chi$ is clearly surjective, it remains to prove that $\chi$ is injective.  Towards this end, suppose that $\varphi(x)$ can be connected to $\varphi(y)$ by an internal path in $\Lambda\backslash \hat{\Gamma}$. Then the same holds for $\theta(x),\theta(y)$ with $\Gamma_{\hyp}$ instead of $\Lambda$. Call $\gamma$ the internal path in $\Gamma_{\hyp}\backslash \hat{\Gamma}$ connecting them. Consider pairs of values $t^j_0,t^j_1$ so that $d_h(\gamma(t^j_i),p)=\nu$, $\gamma(t^j_0)$ is in the same connected component of $\Gamma^*\backslash \mathring{B}(p,\nu)$ as $\gamma(t^j_1)$, and the interval $[t^j_0,t^j_1]$ is maximal among intervals with endpoints satisfying these properties. We can substitute $\gamma|_{[t^j_0,t^j_1]}$ with an internal path in $\Gamma^*\backslash \mathring{B}(p,\nu)$ connecting $\gamma(t^j_0)$ to $\gamma(t^j_1)$ (if $t^j_0>0$ and $t^j_1<1$, otherwise use $x$ and/or $y$), and so we obtain an internal path from $x$ to $y$ in $\Gamma^*\backslash \hat{\Gamma}$.
\end{proof}

\begin{thm}\label{main}
$\Theta|_{\pi_1(|\Gamma|)}:\pi_1(|\Gamma|)\to \pi_1(\Gamma_{\hyp})$ is injective.
\end{thm}

The following corollary will be useful for our applications.

\begin{cor}
\label{factor}
 If $\Lambda$ is a hyperfinite graph and there exist internally continuous maps $\varphi:|\Gamma|^*\to \Lambda$, $\psi:\Lambda\to \Gamma_{\hyp}$ so that $\psi\circ\varphi=\theta$, then the map induced by $\varphi$ at the level of fundamental groups is injective when restricted to $\pi_1(|\Gamma|)$.
\end{cor}

\begin{proof} (of Theorem \ref{main})
  Suppose that $\theta(\alpha)$ is null-homotopic, for some loop $\alpha:[0,1]\to|\Gamma|$ based at $p\in\Gamma$. We will construct a continuous family of paths $\gamma_t$ connecting $p$ to $\alpha(t)$. The idea for constructing such a family is the following. As $\theta(\alpha)$ is null-homotopic, it lifts to a loop in the universal cover $\widetilde{\Gamma_{\hyp}}$ of $\Gamma_{\hyp}$, which is a tree. We can then project on $\Gamma_{\hyp}$ a family of geodesics connecting a lift of $p$ to points in the lift of $\theta(\alpha)$ based at the same point, and use these projected paths to construct the required homotopy.
\par
{\bf Convenient reparameterizations.} Even though it is not absolutely necessary, we will assume that $\alpha$ makes no partial crossing of edges, meaning that if the point $\alpha(x)$ is contained in the interior of an edge then there is an interval $I$ containing $x$ so that $\alpha(I)$ is the edge containing $\alpha(x)$ and $\alpha|_I$ connects the endpoints of the said edge. We can achieve this by a homotopy in view of \cite[Lemma 2]{diestel}. Let $\rho: \n^*\to\n^*$ be the function such that $\rho(\xi)$ is the cardinality of the edges at distance at most $\xi$ from $p\in\Gamma_{\hyp}$ that are crossed by $\theta(\alpha)$, counted with multiplicity. (The distance of an edge from $p$ is the minimal distance of its endpoints from $p$.) Notice that if $\xi$ is finite, then $\rho(\xi)$ is finite as it coincides with the number of edges at distance at most $\xi$ from $p$ crossed by $\alpha$ counted with multiplicity. We can reparameterize $\theta(\alpha):[0,\lambda]\to\Gamma_{\hyp}$, for some finite $\lambda\in \R^*$, so that subintervals of $[0,\lambda]$ which correspond to crossing an edge at distance $\xi$ from $p$ have length $1/(2^{\xi}\rho(\xi))$. Notice that $\lambda$ is indeed finite because, as there are at most $\rho(\xi)$ subpaths of $\theta(\alpha)$ corresponding to crossing an edge at distance $\xi$ from $p$, we have $\lambda\leq\sum_{\xi\in\n^*}\rho(\xi)/(2^\xi\rho(\xi))=1$. Also, we can reparameterize $\alpha$ so that for each $\tau\in[0,\lambda]$ we have $\alpha(\st(\tau))=\st(\theta(\alpha)(\tau))$ if $\theta(\alpha)(\tau)\in\hat{\Gamma}$, and $\alpha(\st(\tau))$ is the end corresponding to $\theta(\alpha)(\tau)$ otherwise (the correspondence is given by Lemma \ref{chi}). Here, $\st:\hat{\Gamma}\to\Gamma$ is the map so that $\st(x)$ has infinitesimal distance from $x$.
\par
 {\bf The homotopy.} As already mentioned, the fact that $\theta(\alpha)$ is null-homotopic means that it can be lifted to a loop $\widetilde{\theta(\alpha)}$ in $\widetilde{\Gamma_{\hyp}}$, based, say, at $\tilde{p}$. Fix $t\in[0,\st(\lambda)]$ and choose any $\tau\in[0,\lambda]$ so that $\st(\tau)=t$. Denote by $\tilde{d}$ the metric of $\widetilde{\Gamma_{\hyp}}$. Let $\delta_{\tau}:[0,\tilde{d}(\tilde{p},\widetilde{\theta(\alpha)}(\tau))]\to\Gamma_{\hyp}$ be the projection to $\Gamma_{\hyp}$ of a geodesic from $\tilde{p}$ to $\widetilde{\theta(\alpha)}(\tau)$. We can reparameterize $\delta_\tau$ to obtain $\phi_\tau:[0,\lambda_\tau]\to\Gamma_{\hyp}$ with the property that subintervals of $[0,\lambda_\tau]$ which correspond to crossing an edge at distance $\xi$ from $p$ have length $1/(2^{\xi}\rho(\xi))$. Notice that the image of $\delta_\tau$ (and so that of $\phi_\tau$) is contained in that of $\theta(\alpha)$. In fact, given any path in a tree, the (image of the) geodesic connecting its endpoints is contained in (the image of) the path. As a consequence, $\phi_\tau$ is the concatenation of subpaths of $\theta(\alpha)$ with disjoint domain and so $\lambda_\tau$ is finite. Finally, define $\gamma_t:[0,1]\to |\Gamma|$ by $\gamma_t(u)=\st(\phi_\tau(u\lambda_\tau))$ if $\phi_\tau(u\lambda_\tau)\in \hat{\Gamma}$, and let $\gamma_t(u)$ be the end corresponding to $\phi_\tau(u\lambda_\tau)$ otherwise (we are using Lemma \ref{chi} again).
\par
We claim that $H(t,u)=\gamma_t(u)$ is continuous and hence provides a homotopy from $\alpha$ to the trivial loop.
\par
{\bf Long common subpaths.} The key observation is that $\phi_{\tau_0},\phi_{\tau_1}$ differ at most in final intervals of the respective domains of length bounded by $|\tau_0-\tau_1|$, meaning $$\phi_{\tau_0}\big|_{\left[0,\lambda_{\tau_0}-|\tau_0-\tau_1|\right]}=\phi_{\tau_1}\big|_{\left[0,\lambda_{\tau_1}-|\tau_0-\tau_1|\right]}.$$
This follows from the fact that geodesics $\gamma_0,\gamma_1$ in a tree starting at the same point and so that the distance between their final points is $d$ share a common initial subgeodesic of length at least $l(\gamma_0)-d$, and the remaining subgeodesics of both $\gamma_0$ and $\gamma_1$ are contained in the geodesic connecting the final points of $\gamma_0,\gamma_1$. In particular, we have that $\delta_{\tau_0}$ shares a common initial subpath of length $l(\delta_{\tau_0})-\tilde{d}(\widetilde{\theta(\alpha)}(\tau_0),\widetilde{\theta(\alpha)}(\tau_1))$ with $\delta_{\tau_1}$. We have to check that when reparameterizing to obtain $\phi_{\tau_i}$, the maximal such subpath $\beta$ becomes a subpath of $\phi_{\tau_i}$ containing $\phi_{\tau_i}\big|_{\left[0,\lambda_{\tau_i}-|\tau_0-\tau_1|\right]}$. This follows essentially from the fact that everything is parameterized ``in the same way.''  Let us make this more precise. Consider the final subpath $\phi'_i$ of $\phi_{\tau_i}$ connecting the final point of $\beta$ to $\widetilde{\theta(\alpha)}(\tau_i)$. Arguing as above, $\phi'_i$ is contained in the projection on $\Gamma_{\hyp}$ of the image of a geodesic from $\widetilde{\theta(\alpha)}(\tau_0)$ to $\widetilde{\theta(\alpha)}(\tau_1)$, which in turn is contained in $\theta(\alpha)([\tau_0,\tau_1])$, and moreover, $\phi'_i$ is the concatenation of subpaths of $\theta(\alpha)|_{[\tau_0,\tau_1]}$ with disjoint domains. So, the domain of $\phi'_i$ has length at most $|\tau_0-\tau_1|$, which implies that the domain of the reparameterization of $\beta$ as a subpath of $\phi_{\tau_i}$ contains $\left[0,\lambda_{\tau_i}-|\tau_0-\tau_1|\right]$, as required.
\par
{\bf Proof of continuity.} The key observation implies that $t\mapsto \st(\lambda_\tau)$ is continuous. Hence we can check continuity of the function $$H':\{(t,u)\in [0,1]\times \R_{\geq 0}:u\leq \st(\lambda_\tau)\}\to |\Gamma|$$ given by $H'(t,u)=\gamma'_{t}(u)$. Here $\gamma'_t$ is defined by $\gamma'_t(u)=\st(\phi_\tau(\sigma))$ if $\phi_\tau(\sigma)\in \hat{\Gamma}$, where $\st(\sigma)=u$ and $\sigma=\sigma(u)\in[0,\lambda_\tau]$, while $\gamma'_{t}(u)$ is the end corresponding to $\phi_\tau(\sigma)$ otherwise. $H$ is the composition of $H'$ and the continuous function $(t,u)\mapsto (t,u\cdot \st(\lambda_{\tau}))$, so that $H$ is continuous if $H'$ is. (Notice that $\phi_\tau(\sigma(u\cdot \st(\lambda_{\tau})))$ is infinitesimally close to $\phi_\tau(u\cdot \lambda_\tau)$ if one of them is in $\hat{\Gamma}$ and that otherwise the subpath of $\phi_\tau$ connecting them does not intersect $\hat{\Gamma}$ so that they correspond to the same end.)
\par
We first show that $H'$ is continuous at $(t,u)$ if $H'(t,u)\in\Gamma$. Indeed, we claim that for $n\geq d(p,\phi_\tau(\sigma(u)))+1$ and $\epsilon<1/2$, if $|t-t'|< \epsilon/(2^{n+1}\rho(n+1))$ and $|u-u'|< \epsilon/(2^{n+1}\rho(n+1))$, then we have $d(H'(t,u),H'(t',u'))\leq 3\epsilon$. To see this, first observe that if $|t-t'|< \epsilon/(2^{n+1}\rho(n+1))$, then $\phi_\tau$ and $\phi_{\tau'}$ either differ by paths contained in $B(p,n+1)$ whose lengths are each bounded above by $\epsilon$ or by paths contained in the complement of $B(p,n-1)$. Indeed, keeping into account our choice of parameterization, if the endpoint of $\phi_\tau$ is in $B(p,n)$, then $\phi_\tau|_{[\max\{0,\lambda_\tau-|\tau-\tau'|\},\lambda_\tau]}$ is contained in $B(p,n+1)$ and hence has length at most $\epsilon$, and otherwise the said subpath cannot travel through an edge at distance $n-1$ from $p$.
To finish the proof of our claim, it suffices to observe that if $\beta$ is any path contained in $B(p,n+1)$ satisfying our choice of the parameterizations (e.g. $\phi_{\tau},\phi_{\tau'}$) and if $|\sigma-\sigma'|< \epsilon/(2^{n+1}\rho(n+1))$, then $d_h(\beta(\sigma),\beta(\sigma'))<\epsilon$. 
\par
Suppose instead that $H'(t,u)$ is an end. We claim that if $|t-t'|< 1/(2^{n+1}\rho(n+1))$ and $|u-u'|< 1/(2^{n+1}\rho(n+1))$, then for $\sigma=\sigma(u)$ and $\sigma'=\sigma'(u)$, $\phi_\tau(\sigma)$ can be connected to $\phi_{\tau'}(\sigma')$ by a path lying outside $B(p,n)$, which proves the continuity of $H'$ at $(t,u)$ in view of the nonstandard characterization of the topology of $|\Gamma|$ and because $\phi_\tau(\sigma)$, $\phi_{\tau'}(\sigma')$ represent the same point of $|\Gamma|$ as $H'(t,u)$, $H'(t',u')$.
\par
With the given restrictions on $t',u'$, we have that $\phi_{\tau},\phi_{\tau'}$ differ by paths contained in some connected component $C_1$ of $\Gamma_{\hyp}\backslash B(p,n)$ and also $\phi_{\tau'}(\sigma)$, $\phi_{\tau'}(\sigma')$ must lie in the same connected component $C_2$ of $\Gamma_{\hyp}\backslash B(p,n)$. If $C_1\neq C_2$ then $\phi_{\tau}$ and $\phi_{\tau'}$ coincide in $C_2$, so the path connecting $\phi_\tau(\sigma)$ to $\phi_{\tau'}(\sigma')$ outside $B(p,n)$ can be taken to be a subpath of $\phi_\tau$. If $C_1=C_2$ then by definition $\phi_{\tau'}(\sigma'), \phi_{\tau'}(\sigma)$ and $\phi_{\tau'}(\sigma),\phi_{\tau}(\sigma)$ lie in the same connected component of $\Gamma_{\hyp}\backslash B(p,n)$. 
\end{proof}



\subsection{The image of $\Theta$}


It would be nice if we could characterize the image of our embedding $\Theta|_{\pi_1(|\Gamma|)}$, but unfortunately we are currently unable to accomplish this goal.  In this subsection, we discuss some of the difficulties we face in this endeavor.

Once again, let us consider the 1-way sideways infinite ladder $\Gamma$ from the Introduction.  Below, in Figure \ref{sptree}, we have $\Gamma_{\hyp}$ with its internal spanning tree and chords $e_1,e_2,\ldots,e_\nu$.

\begin{figure}[h]
\includegraphics[scale=0.8]{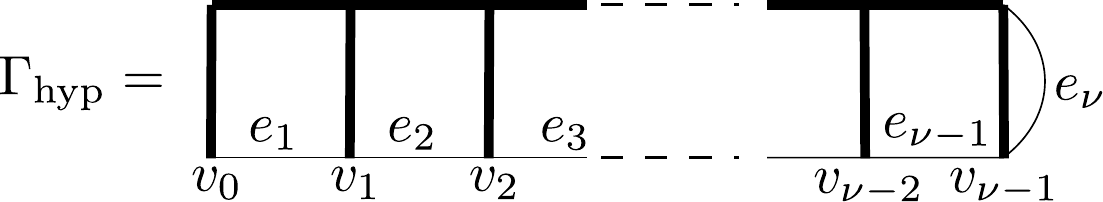}
\caption{}
\label{sptree}
\end{figure}
 The question we face is which internal words are in the image of the embedding $\Theta:\pi_1(|\Gamma|)\to \pi_1(\Gamma_{\hyp})$?  Here is a word that is in the image of this embedding: $e_1e_2\cdots e_\nu$.  Indeed, this word arises from the loop $\alpha$ that travels down the bottom edge of the ladder to the unique end and then returns along the top edge of the ladder.  However, here is a word that is \emph{not} in the image of the embedding:  $e_1e_2\cdots e_{\nu-1}$.  Indeed, in some sense, this word should result from the same loop that gave rise to $e_1e_2\cdots e_\nu$, but it is impossible for the standard loop to ``stop'' at $e_{\nu-1}$ and continue upwards.  More generally, for any $\eta\in \n^*\setminus \n$ with $\eta<\nu$, the word $e_1e_2\cdots e_{\eta}$ should not be in the image of the embedding.

We also need to ensure that words representing ``loops at infinity'' are not in the image of our embedding.  For example, consider the word $$e_1e_2\cdots e_{\nu-1}e_\nu e_{\nu-1}^{-1}\cdots e_2^{-1}e_1^{-1}.$$  Indeed, if this word arose from some loop $\alpha$, then by transfer, $\alpha$ would have made a loop at some finite stage $n\in \n$ and the image of $[\alpha]$ under our embedding would have been $e_1e_2\cdots e_{n-1}e_ne_{n-1}^{-1}\cdots e_2^{-1}e_1^{-1}$.  With these hurdles in mind, we pose the following

\begin{question}
What is the image of the embedding $\Theta:\pi_1(|\Gamma|)\to \pi_1(\Gamma_{\hyp})$?
\end{question}


\subsection{Variations on the construction of $\Gamma_{\hyp}$}

The combinatorial characterization of $\pi_1(|\Gamma|)$ given in \cite{diestel} also applies to certain subspaces of $|\Gamma|$, the so-called \emph{standard subspaces} of $|\Gamma|$.  The subspace $H$ of $|\Gamma|$ is called \emph{standard} if $H$ is closed, connected, and contains every edge of which it contains an inner point.  We leave it to the reader to check that our construction is readily adaptable to the case of a standard subspace of $|\Gamma|$.

For applications Section \ref{appl}, it will be important to note that, in view of Corollary \ref{factor}, we can construct $\Gamma_{\hyp}$ in
the following way. Let $\{\Delta_i\}$ be an internal collection of
disjoint, internally connected subgraphs of $\Gamma^*$ disjoint from $\Gamma_{\f}$ such that there are hyperfinitely many vertices not contained
in any $\Delta_i$; we refer to such a collection as a \emph{good collection of subgraphs of $\Gamma^*$}. We then let $\theta$ be the map that collapses each $\Delta_i$ to a point and set $\Gamma_{\hyp}=\theta(\Gamma^*)$.  The above construction of $\Gamma_{\hyp}$ is the special case when we consider the $\Delta_i$'s to be the internally connected components of $\Gamma^*\setminus B(p,\nu)$; we will refer to this construction of $\Gamma_{\hyp}$ as the \emph{standard model} for $\Gamma_{\hyp}$.

\section{Connection to the results in \cite{diestel}}\label{recover}

\subsection{Embedding into an inverse limit of free groups}

\noindent In this subsection, we observe that $\Theta:\pi_1(|\Gamma|)\to \pi_1(\Gamma_{\hyp})$ factors through an injective map $\Psi:\pi_1(|\Gamma|)\to \varprojlim F_n$ for certain finitely generated free groups $F_n$. For $n\in\n$, set $\theta_n:\Gamma\to \Gamma_n$ to be the map collapsing each connected component of $\Gamma\setminus \mathring{B}(p,n)$ to a point. Set $F_n:=\pi_1(\Gamma_n)$.  There are natural continuous maps $\rho^m_n:\Gamma_m\to\Gamma_n$ when $m\geq n$, so that $\{F_n\}$ is an inverse system of groups. By Lemma \ref{ext}, we have a continuous extension of $\theta_n$ to a map $\theta_n:|\Gamma|\to \Gamma_n$ satisfying $\theta_n=\rho^m_n\circ \theta_m$ whenever $m\geq n$.  We thus have a map $\Psi:\pi_1(|\Gamma|)\to \varprojlim F_n$.

Suppose we choose our nonstandard extension $\Gamma^*$ to be $\Gamma^\n/\mu$ for some non-principal ultrafilter $\mu$ on $\n$, and choose $\nu=[(\operatorname{id}:\n\to\n)]_\mu$.  Then $\Gamma_{\hyp}=\prod_\mu \Gamma_n$.  Consequently, $\pi_1(\Gamma_{\hyp})=\prod_\mu F_n$. Now consider the map $\Phi:\varprojlim F_n\to \prod_\mu F_n$ given by $\Phi((x_n))=[(x_n)]_\mu$. Then $\Phi$ is injective:  if $x_n=y_n$ $\mu$-a.e., then since $\mu$ is nonprincipal, $x_n=y_n$ for infinitely many $n$.  Since $(x_n),(y_n)\in \varprojlim F_n$, it follows that $x_n=y_n$ for all $n$.  

We leave it to the reader to check that $\Theta=\Phi\circ\Psi$.  It follows that $\Psi$ is injective.

\subsection{The Standard Description of $\pi_1(|\Gamma|)$}

After seeing an initial draft of this article, Diestel asked us whether we could recover the main theorem of \cite{diestel}, namely Theorem 15.  After some further effort, we can indeed accomplish this.  First, let us briefly describe the content of Theorem 15.  Fix a topological spanning tree $T$ of $|\Gamma|$, namely an arc-connected, closed subspace of $|\Gamma|$ that does not contain any circle (that is, a homeomorphic image of $\mathbb{S}^1$) and which contains every edge of which it contains an interior point.  We let $\{e_i \ : \ i\in \n\}$ enumerate the chords of $T$, namely the edges of $\Gamma$ not in $T$, equipped with a fixed orientation.  (We may assume that there are infinitely many chords, for otherwise $|\Gamma|$ is homotopy equivalent to a finite graph and there is nothing to prove.)

Given a loop $\alpha$ in $|\Gamma|$ based at $p$, one can consider the trace of $\alpha$ on the chords, giving a countable ``word'' $w_\alpha$, which should be viewed as a function from some countable linearly ordered set $S$ to the set $A:=\{e_i^{\pm 1} \ : \ i\in \n\}$.  The authors of \cite{diestel} define a suitable notion of reduction of words and prove that every word $w$ reduces to a unique reduced word $r(w)$.  They let $F_\infty$ denote the set of reduced words and endow it with the obvious group multiplication (namely $r(w)\cdot r(w')$ is the unique reduced word associated to the concatenation of $w$ and $w'$).  They prove that if $\alpha$ and $\beta$ are homotopic loops based at $p$, then $r(w_\alpha)=r(w_\beta)$, whence they obtain a map $[\alpha]\mapsto r(w_\alpha):\pi_1(|\Gamma|)\to F_\infty$.  Furthermore, they prove that this map is \emph{injective} and characterize the image of this map as precisely the set of words whose monotonic subwords converge (we will explain this terminology later in the next subsection).  They then show that $F_\infty$ embeds into an inverse limit of finitely generated free groups by mapping $r(w)$ to the family $(r(w\upharpoonright I))_{I\subseteq S \text{ finite}}$.  They characterize the image of this latter embedding as the set of elements of the inverse limit such that each letter appears only finitely often.  As we remarked in the introduction, this latter embedding is purely algebraic and the nonstandard set-up has nothing to say about this portion of the theorem.

By far the hardest part of the paper \cite{diestel} is their Lemma 14, which shows that if $r(w_\alpha)=\emptyset$ for a loop $\alpha$ based at $p$, then $\alpha$ is nullhomotopic, implying the injectivity of the map $\pi_1(|\Gamma|)\to F_\infty$.  In what follows, we use our framework to give another proof of Lemma 14.  Although our proof of Lemma 14 appears quite involved (due to the nebulous relationship between topological spanning trees in $|\Gamma|$ and internal spanning trees in $\Gamma_{\hyp}$), our proof is still significantly shorter than the standard proof, which is 14 pages in length.

\subsection{Recovering The Standard Description}

We are going to build a hyperfinite graph $\Lambda$, an internal spanning tree $T'$ of $\Lambda$, and an internally continuous map $\varphi:|\Gamma|^*\to \Lambda$. The construction will be carried out in such a way that if $e$ is a chord of $T^*$, then either $\varphi$ maps $e$ to a point in $\Lambda$ or else it maps $e$ to a chord of $T'$, which can and will be identified with $e$. Every step of our construction has been conceived for this to obviously stay true, so we will make no further mention of it during the construction. 

We start with $\Lambda_0:=B(p,\nu)$, viewed as an internal subgraph of $\Gamma^*$, and $T_0:=T^*\cap \Lambda_0$.  Let $w_1,\ldots,w_M$ enumerate all pairs of points of distance $\nu$ from $p$.  If $T_0$ is already internally connected, then set $\Lambda^\#:=\Lambda_0$ and $T^\#:=T_0$.  Otherwise, there must be $i\in \{1,\ldots,M\}$ such that, adding an edge at $w_i$ to both $\Lambda_0$ and $T_0$ yields no internal cycles in $T_0\cup \{w_i\}$.  Indeed, if this is not true, then $T_0$ was already internally connected:  fix $x,y\in T_0$ and take paths in $T_0$ connecting $x$ to $x'$ and $y$ to $y'$, where $x',y'$ have distance $\nu$ from $p$.  By assumption, adding an edge connecting $x'$ to $y'$ creates an internal loop, meaning that there is a path in $T_0$ connecting $x'$ to $y'$, whence there is a path in $T_0$ connecting $x$ to $y$.  Let $i_0\in \{1,\ldots,M\}$ be minimal such that adding an edge at $w_i$ creates no internal cycles in $T_0\cup \{w_i\}$.  Set $\Lambda_1:=\Lambda_0\cup \{w_{i_0}\}$ and $T_1:=T_0\cup \{w_{i_0}\}$.  

If $T_1$ is internally connected, then we set $\Lambda^\#:=\Lambda_1$ and $T^\#=T_1$.  Otherwise, we continue as above:  there must be $i\in \{i_0+1,\ldots,M\}$ such that adding an edge at $w_i$ creates no new internal cycles.  Indeed, suppose this was not true.  Fix $x,y\in T_1$ and let $x'$ and $y'$ be as before.  If $\{x',y'\}=w_{i_0}$, then $x$ and $y$ are connected in $T_1$.  Otherwise, $\{x',y'\}=w_i$ for some $i\neq i_0$.  If $i<i_0$, then by the previous paragraph, $x$ and $y$ are connected in $T_0\subseteq T_1$.  If $i>i_0$, then by our assumption, adding an edge at $w_i$ creates an internal cycle, meaning that $x'$ and $y'$ are connected in $T_1$, whence $x$ and $y$ are connected in $T_1$.  Set $i_1\in \{i_0+1,\ldots,M\}$ to be minimal with this property and set $\Lambda_2:=\Lambda_1\cup \{i_1\}$ and $T_2:=T_1\cup \{i_1\}$.

We continue in this fashion, ending up with a hyperfinite graph $\Lambda^\#$ and an internal spanning tree $T^\#$ of $\Lambda^\#$.  $\Lambda$ and $T'$ are going to be obtained by adding further vertices and edges as we now explain.

As the first part of $\varphi$, we collapse all connected components of $\Gamma^*\setminus B(p,\nu+1)$ to points and extend this map to $|\Gamma|^*$ (see Lemma \ref{ext}), where now we think of these hyperfinitely many points as points of distance $\nu+1$ from $p$.  As such, there are hyperfinitely many edges $z_1,\ldots,z_N$ of $\Gamma^*$ connecting points of radius $\nu$ to these points of radius $\nu+1$.  Here, if $x$ and $y$ have distance $\nu+1$ from $p$, are collapsed to the same point, and are both connected to some vertex of radius $\nu$, then we include both edges in our enumeration $z_1,\ldots,z_N$.  Let us write $z_i:=(x_i,y_i)$, with $x_i$ of distance $\nu$ and $y_i$ of distance $\nu+1$.


We now show how to inductively deal with the $z_j$'s.  We split up into two cases.

\

\noindent \textbf{Case 1:}  No edge containing $y_j$ has been previously collapsed to $x_j$ (e.g., for $j=1$).  If $z_j\in T^*$, then collapse it to $x_j$.  Otherwise, suppose $z_j\notin T^*$.  We then add $z_j$ to $\Lambda$ and $T'$, and have $\varphi$ act identically on $z_j$.

\

\noindent \textbf{Case 2:}  Some edge containing $y_j$ has been previously collapsed to $x_j$.  First suppose that $z_j\in T^*$.  We then have $\varphi$ map $z_j$ to the unique path in $T'$ connecting $x_j$ to $y_j$ and do nothing to $\Lambda$.  It is extremely important to observe that this path does not enter the finite portion of $T'$, for otherwise $T$ would contain a circle (either consider the ``trace'' in $|\Gamma|$ of a reparameterization of the said path or notice that there would exist $n$ so that for all $r$ there are two points in $T$ at distance at least $r$ from $p$ so that the path connecting them in $T$ intersects $B(p,n)$, and use Arzel\`a-Ascoli).  We now suppose that $z_j\notin T^*$.  Then we add $z_j$ as a new edge to $\Lambda$, but add nothing to $T'$.  We observe that $T'$ is still a spanning tree of $\Lambda$.

\

In this way, we obtain our hyperfinite graph $\Lambda$, our internal spanning tree $T'$ of $\Lambda$, and our internally continuous map $\varphi:|\Gamma|^*\to \Lambda$.  We observe that the chords of $T'$ are also chords of $T^*$.

Let $\nu'$ be the minimal distance from $p$ of vertices involved in the cases where edges $z_i$ were mapped by $\varphi$ to paths in $T'$; observe that $\nu'\in \n^*\setminus \n$.  Let $\Gamma_{\hyp}$ be as before except using $\nu'$ rather than $\nu$.  Let $\theta:|\Gamma|^*\to \Gamma_{\hyp}$ be as before and observe that $\theta$ factors as $\theta=\rho\circ \varphi$, where $\rho:\Lambda\to \Gamma_{\hyp}$ is the natural collapsing map.  Since $\theta$ is $\pi_1$-injective, it follows that $\varphi$ is $\pi_1$-injective (Corollary \ref{factor}).

We are finally ready to prove Lemma 14 from \cite{diestel}.

\begin{prop}[\cite{diestel}, Lemma 14]
Suppose that $\alpha:[0,1]\to |\Gamma|$ is such that $w_\alpha$ reduces to the empty word.  Then $\alpha$ is nullhomotopic.
\end{prop}

\begin{proof}
Let $I\subseteq \n^*$ be such that $i\in I \leftrightarrow e_i\in E(\Lambda)$, so $I$ is hyperfinite.  Set $w:=w_\alpha$ and say $w:S\to A$, where $S$ is some countable, linearly ordered set.  We have a reduction $R$ of $w$ to the empty word, whence $R^*$ is an internal reduction of $w^*$ to the empty word.  By (1) on page 8 of \cite{diestel}, $R_{\Lambda}:=\{(s,s')\in R^* \ | \ w(s)\in A_I\}$ is an internal reduction of $w^*\upharpoonright I$, and this reduction must be to the empty word as well.  Since $w^*\upharpoonright I$ is a hyperfinite word, $w^*\upharpoonright I$ internally reduces (in the usual sense) to the empty word.  Since $w^*\upharpoonright I$ is the trace of $\varphi(\alpha)$ on the chords of $T'$, we see that $\varphi(\alpha)$ is internally nullhomotopic, whence $\alpha$ is nullhomotopic. 
\end{proof}

We now proceed to prove the other difficult part of Theorem 15, namely the part which characterizes the image of the embedding $\pi_1(|\Gamma|)\to F_\infty$.  To explain this, we need some terminology.  First, given a word $w:S\to A$, we say that a subword $w\upharpoonright S':S'\to A$ is \emph{monotonic} if $S'$ is infinite and we can write $S':=\{s_0,s_1,s_2,\ldots,\}$ where either $s_0<s_1<s_2<\cdots$ or $s_0>s_1>s_2>\cdots$.  We say that the monotonic subword $w\upharpoonright S'$ of $w$ \emph{converges} if there exists an end $\e(r)$ of $|\Gamma|$ such that, whenever we choose $x_n\in w(s_n)$, the sequence $(x_n)$ converges to $\e(r)$.  It is very easy to see that, for any loop $\alpha$ in $|\Gamma|$ based at $p$, every monotonic subword of the word $w_\alpha$ converges.

\begin{prop}[Part of \cite{diestel}, Theorem 15]
Suppose that the word $w:S\to A$ has the property that all its monotonic subwords converge. Then there exists a loop $\alpha$ so that $w=w_\alpha$. 
\end{prop}

\begin{proof}
 We fix $\Lambda$ and $T'$ as above.  As in the proof of the previous proposition, we consider the hyperfinite word $w^*\upharpoonright I$.  As in the case of finite graphs, $w^*\upharpoonright I$ corresponds to an internal loop $\beta$ in $\Lambda$ based at $p$ obtained by concatenating alternatively chords of $T'$ and injective paths in $T'$.
 
 \
 
\noindent \textbf{Claim:}  $\beta$ crosses each edge of $\Gamma$ finitely many times. 

\

\noindent For the moment, suppose that the Claim holds.  As we've done several times, we consider $\rho:\n^*\to \n^*$, where $\rho(\xi)$ is the the number of edges at distance at most $\xi$ from the basepoint $p$ that $\beta$ crosses, counted with multiplicity. The claim implies that $\rho(n)\in \n$ for $n\in \n$. We then reparameterize $\beta$ in such a way that intervals in its domain corresponding to crossing an edge at distance $\xi$ have length $2^{-\xi}\rho(\xi)$ (and such crossings are performed linearly). Consider the path $\alpha$ defined by $\alpha(t)=\st(\beta(t))$ if $\st(\beta(t))$ is defined, while $\alpha(t)$ is the end corresponding to $\beta(t)$ otherwise (we are using Lemma \ref{chi}). Arguments identical to those earlier in the paper show that $\alpha$ is a loop in $|\Gamma|$ based at $p$ and that $w=w_\alpha$, as we desired.

\

\noindent \textbf{Proof of Claim:}  Fix $e\in E(\Gamma)$.  We must show that $\beta$ crosses $e$ finitely many times.  This is easily seen to be the case if $e$ is a chord of $T$.  Indeed, we then have $e$ and $e^{-1}$ appear only finitely many times in $w$, for otherwise, by Ramsey's theorem, there would exist a monotonic subword $w\upharpoonright S'$ of $w$ such that $w(s)\in \{e,e^{-1}\}$ for all $s\in S'$; it is clear that $w\upharpoonright  S'$ does not converge.  Consequently, $e$ and $e^{-1}$ appear in $w^*\upharpoonright I$ only finitely many times, whence $\beta$ crosses $e$ only finitely many times.

We now suppose that $e\in T$ and further suppose, towards a contradiction, that $e$ is crossed infinitely many times by $\beta$.  First observe that $T'\setminus \mathring{e}$ has two components, say $T_1$ and $T_2$.  Given a chord of $T'$, we will say it is of type $(i,j)$, where $i,j\in \{1,2\}$, if its initial point is in $T_i$ and its final point is in $T_j$.  We can then find a hyperfinite sequence $s_0,s_1,\dots,s_\xi$, where $s_i\in S^*$ and $\xi>\n$, so that going in $T'$ from the final point of $w(s_i)$ to the initial point of $w(s_{i+1})$ requires crossing $e$, whence $w(s_i)$ is of type $(*,j)$ and $w(s_{i+1})$ is of type $(i,\dagger)$, where $i\not=j$.  

We claim that there is a infinite, hyperfinite subword of $w(s_0)w(s_1)\cdots w(s_\xi)$ that is either \emph{uniform} in the sense that every $w(s_i)$ has type $(1,2)$ or every $w(s_i)$ has type $(2,1)$, or \emph{alternating} in the sense that, if $w(s_i)$ is of type $(1,1)$ (resp. $(2,2)$), then $w(s_{i+1})$ is of type $(2,2)$ (resp. $(1,1)$).
Indeed, either the subword consisting of chords of type $(1,2)$ is infinite or the subword consisting of chords of type $(2,1)$ is infinite, or else the subword obtained by removing all chords of type $(1,2)$ and $(2,1)$ is infinite and hyperfinite.  We may thus suppose that $w(s_0)w(s_1)\cdots w(s_\xi)$ is either uniform or alternating.  However, it now follows that $w(s_0)w(s_1)\cdots w(s_\xi)$ is either uniform or alternating in $T^*$.  (Recall that chords of $T'$ are also chords of $T^*$.)  Indeed, if the unique path in $T^*$ connecting the final point of $w(s_i)$ to the initial point of $w(s_{i+1})$ did not pass through $e$, then after applying $\varphi$, we would see that, since $\varphi$ acts identically on $\hat{\Gamma}$, the unique path in $T'$ connecting the final point of $w(s_i)$ to the initial point of $w(s_{i+1})$ does not pass through $e$.  It remains to note that the unique paths in $T^*$ and $T'$ connecting the final point of $w(s_i)$ to the initial point of $w(s_{i+1})$ agree on $\hat{\Gamma}$.   


\par
Thus, by transfer, there exist arbitrarily long finite subwords $w\upharpoonright S'$ that are either uniform for $e$ or alternating for $e$. If we have arbitrarily long finite uniform subwords, then we have an infinite uniform subword obtained by taking the union of the finite subwords.  To obtain an infinite alternating subword from arbitrarily long finite alternating subwords is a little more subtle. Let $w_n=w\upharpoonright S_n$ be subwords alternating for $e$ of length $10^{n+1}$. Start with $S'_0=S_0$. Let $I^0_1,\dots,I^0_{j(0)}$ be the natural intervals into which $S$ is subdivided by removing the elements of $S'_0$. Then there exists $j$ so that $I^0_j$ contains at least, say, $4$ elements of $S_1\backslash S'_0$.  We can extend $S'_0$, obtaining $S'_1$, by adding at least $2$ such elements in a way that $w\upharpoonright S'_1$ is again alternating for $e$. Proceeding in this way, we get \emph{nested} arbitrarily long subwords $w\upharpoonright S'_n$ alternating for $e$. 

In either case, after applying Ramsey's theorem, we have an infinite, monotonic subword of $w$ such that there are infinitely many endpoints of chords in $T_1$ and infinitely many endpoints of chords in $T_2$.  Such a subword does not converge to an end as every end has a neighborhood that is disjoint from one of the components of $T\setminus \mathring{e}$ (for otherwise $T$ contains a circle).
\end{proof}

\section{Applications}
\label{appl}

\subsection{$\pi_1$-injectivity}
Both of the applications in this subsection rely on applying the transfer principle to the following
\begin{fact}\label{inclusion}
If $\Gamma^1$ is a connected subgraph of the finite, connected graph $\Gamma^2$, then the natural map $\pi_1(\Gamma^1)\to \pi_1(\Gamma^2)$ is injective.  
\end{fact}
To see this, take a spanning tree $T^1$ of $\Gamma^1$ and extend it to a spanning tree $T^2$ of $\Gamma^2$.  Using these spanning trees in the combinatorial characterization of $\pi_1(\Gamma^1)$ and $\pi_1(\Gamma^2)$ yields the aforementioned result.

\begin{prop}
For any locally finite graph $\Gamma$, the inclusion $\Gamma'\to
|\Gamma|$ is $\pi_1$-injective for every subgraph $\Gamma'$ of $\Gamma$.
\end{prop}

\begin{proof}
Let $\alpha$ be a non-trivial loop in $\Gamma'$. Then $\alpha$ is
non-trivial in a finite connected subgraph $\Delta$ of $\Gamma$. Hence
$\theta(\alpha)=\alpha$ is non-trivial in $\Delta^*$ and we can conclude
because the inclusion $\Delta^*\to \Gamma_{\hyp}$ is $\pi_1$-injective by Fact \ref{inclusion}.
\end{proof}

\begin{prop}[\cite{diestel}, Corollary 16]
Suppose that $\Gamma^1$ is a connected subgraph of the locally finite, connected graph $\Gamma^2$. Then the inclusion
$\iota:\Gamma^1\to\Gamma^2$ induces an injective map
$\iota_*:\pi_1(|\Gamma^1|)\to \pi_1(|\Gamma^2|)$.
\end{prop}

\begin{proof}
 Let $\{\Delta_{i}\}$
be the collection of all internally connected components of
$(\Gamma^1)^*\cap\left((\Gamma^2)^*\backslash\mathring{B}_{\Gamma^2}(p,\nu)\right)$.
Then $\{\Delta_i\}$ is a good collection of subgraphs of $\Gamma_1^*$, whence we can construct
$\Gamma^1_{\hyp}$ using this collection.  We use the standard model for $\Gamma^2_{\hyp}$.  With the given models for
$\Gamma^1_{\hyp},\Gamma^2_{\hyp}$, we have a natural map $\iota_{\hyp}:
\Gamma^1_{\hyp}\to\Gamma^2_{\hyp}$ induced by $\iota$, namely $\iota_{hyp}(\theta_1(x))=\theta_2(\iota(x))$, where $\theta_i:(\Gamma^i)^*\to
\Gamma^i_{\hyp}$ are the maps collapsing the relevant subgraphs. This map
satisfies $(\iota_{\hyp})_*\circ\Theta_1=\Theta_2\circ\iota_*$ at the level of the
fundamental groups, as can easily be checked. It suffices to show
that $\iota_{\hyp}$ is $\pi_1$-injective. Indeed, consider the graph
$\Lambda$ obtained from $\Gamma^1_{\hyp}$ by identifying
$x,y\in\Gamma^1_{\hyp}$ if $\iota_{\hyp}(x)=\iota_{\hyp}(y)$. Notice that if
$x\neq y$ and $\iota_{\hyp}(x)=\iota_{\hyp}(y)$, then $x$ and $y$ correspond
to subgraphs $\Delta_{i_x},\Delta_{i_y}$ contained in the same connected
component of $(\Gamma^2)^*\backslash\mathring{B}_{\Gamma^2}(p,\nu)$. In
particular, $\Lambda$ is obtained from $\Gamma^1_{\hyp}$ identifying
pairs of vertices, so that the natural map $\Gamma^1_{\hyp}\to \Lambda$
is $\pi_1$-injective. Also, $\Lambda$ is a subgraph of
$\Gamma^2_{\hyp}$, so that the conclusion follows from Fact \ref{inclusion}.
\end{proof}

We should remark that the proof of the preceding result appearing in \cite{diestel} uses the nontrivial fact that every closed, connected subspace of $|\Gamma|$ is arc-connected; our proof avoids using this nontrivial result and is essentially an application of the transfer principle and our construction.

Notice that the proposition implies, as in \cite[Corollary 18]{diestel}, that $\pi_1(|\Gamma|)$ is free is and only if every end has a contractible neighborhood (such ends are called \emph{trivial}). Namely, if all ends are trivial, then $|\Gamma|$ is homotopy equivalent to a finite graph, whereas if $|\Gamma|$ contains a non-trivial end, then it contains a subgraph $\Gamma'$ with exactly one non-trivial end. The compactification of $\Gamma'$ is homotopy equivalent to the Hawaiian Earring and hence its fundamental group, which is contained in $\pi_1(|\Gamma|)$, is not free \cite{caco,hig}.

\subsection{Homology}
Recall that the first homology group (with coefficients in $\z$) $H_1(G)$ of the group $G$ is its abelianization, i.e. the quotient of $G$ by the subgroup $[G,G]$ generated by commutators of elements of $G$. The \emph{commutator length} of $g\in G$ is defined to be infinite if $g\notin [G,G]$, while it is the minimal integer $n$ so that $g$ can be written as the product of $n$ commutators if $g\in[G,G]$.
\par
As usual, let $\Gamma$ be a connected locally finite graph, and let $\theta:|\Gamma|^*\to \Gamma_{\hyp}$ be the map we defined in Section \ref{emb}.

\begin{lemma}\label{homol}
 If the loop $\alpha$ is null-homologous, i.e. it represents $0\in H_1(\pi_1(|\Gamma|))$, then $\theta(\alpha)$ has finite commutator length as an element of $\pi_1(\Gamma_{\hyp})$.
\end{lemma}

\begin{proof}
 Let $g\in \pi_1(|\Gamma|)$ be the element represented by $\alpha$ and let $f:G\to H_1(G)$ be the natural map. If $f(g)=0$, then we can write $g$ as the product of, say, $n$ commutators. As $\Theta:\pi_1(|\Gamma|)\to \pi_1(\Gamma_{\hyp})$ is a group homomorphism, we have that $\Theta(g)$ can be written as a product of $n$ commutators as well.
\end{proof}

The lemma can be used to find ``unexpected'' non-null-homologous loops, as the commutator length of elements of a free group can be computed, see \cite{GT-commlength} and below. For example, it is shown in \cite[Section 6]{DiSp-hom} that, among others, the following loop $\alpha$ in the infinite ladder, see Figure \ref{notnullhom} below, is not null-homologous (this loop is the same as the loop $\rho$ as in \cite[Figure 5.3]{DiSp-topoappr3}). The reason why this is interesting, as explained in the introduction, is that the fact that this loop is not null-homologous implies that the usual singular homology for $|\Gamma|$ is different from the topological cycle space of $|\Gamma|$.

\begin{figure}[h]
 \includegraphics[scale=0.9]{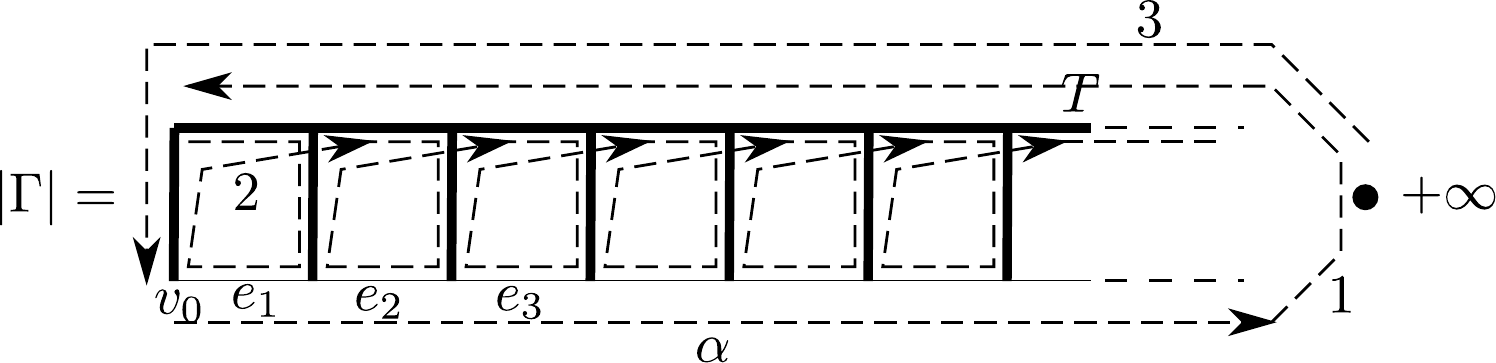}
\caption{The three parts of $\alpha$}
\label{notnullhom}
\end{figure}
$\alpha$ is based at $v_0$, travels through $e_1,e_2\dots$ until it reaches the end, comes back at the top traveling in $T$, then loops around the first square in the clockwise direction, moves in $T$, loops around the second square and so on until it reaches the end. Then, it comes back at the top of the ladder staying in $T$. If the ladder was finite, an analogous loop would be null-homologous as, for each $i$, the sum of the exponents of the occurrences of $e_i$ is $0$.

Let us show that $\alpha$ is not null-homologous using Lemma \ref{homol}.  First, let us recall how commutators length can be computed \cite{GT-commlength}. Let $w$ be a word on the letters $x_1^{\pm 1},\dots ,x_n^{\pm 1}$ representing an element in the commutator subgroup of the free group with free basis $x_1,\ldots,x_n$. The sum of the exponents of each letter is $0$, so that we can choose a \emph{pairing} $\mathcal{P}$ of occurrences of letters with opposite exponent. The \emph{circle graph} associated to such a pairing is a circle $C$ with vertices for each letter of $w$ and appearing in the same cyclic order, together with edges $\{e_i\}$ joining paired letters.
\begin{figure}[h]
 \includegraphics[scale=0.8]{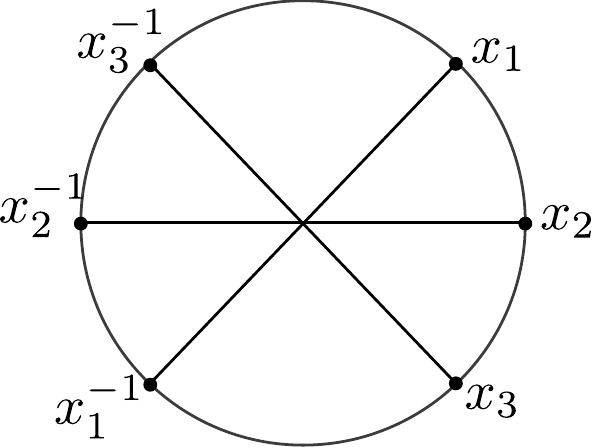}
\caption{Circle graph associated to the unique pairing for $w=x_1x_2x_3x_1^{-1}x_2^{-1}x_3^{-1}$}
\end{figure}

We can associate to such a graph a matrix $M_\mathcal{P}$ whose $ij-$th entry is $1$ if $e_i$ and $e_j$ are linked, meaning that the endpoints of $e_i$ lie in different connected components of $C\backslash\partial e_j$, and is $0$ otherwise (diagonal elements are $0$). It turns out that the commutator length of $w$ is the minimum, over all pairings $\mathcal{P}$, of $\rank(M_\mathcal{P})/2$, where $M_\mathcal{P}$ is regarded as a matrix with coefficients in $\mathbb{Z}/2\mathbb{Z}$.
\par
 If we choose the spanning tree for $\Gamma_{\hyp}$ as in Figure \ref{sptree}, then $\theta(\alpha)$ represents the word $e_1\dots e_\nu e_1^{-1}\dots e_\nu^{-1}$ (notice in particular that $\theta(\alpha)$ is internally null-homologous, as for each $i$ the sum of the exponents of the occurrences of $e_i$ is $0$). It follows from the discussion above, in view of the fact that there is only one possible pairing, that the commutator length of such a word coincides with $\rank(M_\nu)/2$, where $M_\nu$ is the $\nu\times\nu$ matrix whose diagonal elements are $0$ and whose non-diagonal elements are $1$.
It is easy to show inductively that, for $n\in \n$, the matrix $M_n$ defined analogously to $M_\nu$, regarded as a matrix with coefficients in $\mathbb{Z}$, has determinant $(-1)^{n-1}\cdot (n-1)$ (subtracting from the first column of $M_n$ the sum of the other columns divided by $n-2$ yields a matrix where the only non-zero element in the first column is the top entry and is $-(n-1)/(n-2)$).
Consequently, if $n$ is even, then the rank of $M_n$ is $n$.  If $n$ is odd, then since $M_{n-1}$ is a minor of $M_n$, we have that the rank of $M_n$ is $n-1$.  Thus, the commutator length of $e_1\cdots e_\nu e_1^{-1}\cdots e_\nu^{-1}$ is $\rank(M_\nu)/2=\lfloor\nu/2\rfloor>\n$ and, by Lemma \ref{homol}, $\alpha$ is not null-homologous.
\newpage
\begin{nrmks}

\

\begin{enumerate}
\item There are two versions of the topological cycle space, one with coefficients in $\z/2\z$, the other one with coefficients in $\z$. According to \cite{diestel2}\cite[Theorem 5]{DiSp-hom}, a loop $\alpha$ gives rise to the trivial element of the topological cycle space if and only if it traverses every edge an even number of times in the first case, or it traverses every edge the same number of times in both directions in the second case. It is readily seen that this happens if and only if $\theta(\alpha)$ is trivial in $H_1(\Gamma_{\hyp},\z/2\z)$ in the first case, and in $H_1(\Gamma_{\hyp},\z)$ in the second case.
\item Observe that in the above argument, we did not use that $\Theta$ was injective.  This fact is compatible with the standard proof that $\alpha$ is not nullhomologous appearing in \cite{DiSp-hom} as the injectivity of the corresponding map $\pi_1(|\Gamma|)\to F_\infty$ was not used either.
\item It is interesting to note that for any loop $\alpha:[0,1]\to |\Gamma|$, if $\Theta(\alpha)$ is internally nullhomotopic, then $\alpha$ is nullhomotopic, while the corresponding statement for homology is not true, as witnessed by the above discussion.
\item In \cite{diestelsurvey2}, a proof is given that the finite version of the loop $\alpha$ discussed above (called $\rho$ there) is nullhomologous.  Their proof for the finite version is more topological than our algebraic proof and they remark ``But we cannot imitate this proof for $\rho$ and our infinite ladder $L$, because homology classes in $H_1(|G|)$ are still finite chains:  we cannot add infinitely many boundaries to subdivide $\rho$ infinitely often.''  In our opinion, this kind of quote is exactly the kind of thought process that makes nonstandard methods so powerful. 
\end{enumerate}
\end{nrmks}

\bibliographystyle{alpha}
\bibliography{bibl}

\end{document}